\documentclass[a4paper,10pt]{amsart}

\usepackage{graphicx}

\usepackage{cite}
\usepackage{todonotes}
\usepackage{hyperref}

\hypersetup{%
pdftitle={},
pdfsubject={Mathematics},
pdfauthor={Manfred Madritsch, Robert Tichy},
pdfkeywords={}
hyperindex=true,plainpages=false}

\usepackage{a4wide}

\usepackage{amsmath}
\usepackage{amsfonts}
\usepackage{amssymb}
\usepackage{amsthm}
\usepackage{mathtools}

\usepackage{stmaryrd}

\newtheorem{lem}{Lemma}[section]
\newtheorem{thm}[lem]{Theorem}
\newtheorem{prop}[lem]{Proposition}

\numberwithin{equation}{section}

\newtheorem*{cor*}{Corollary}
\newtheorem*{thm*}{Theorem}

\theoremstyle{definition}
\newtheorem{defi}{Definition}[section]

\theoremstyle{remark}

\allowdisplaybreaks[3]

\newcommand{\NN}{\mathbb{N}}
\newcommand{\ZZ}{\mathbb{Z}}
\newcommand{\QQ}{\mathbb{Q}}
\newcommand{\RR}{\mathbb{R}}

\renewcommand{\lVert}{\left\Vert}
\renewcommand{\rVert}{\right\Vert}

\newcommand{\abs}[1]{\left| #1 \right|}
\newcommand{\norm}[1]{\left\| #1 \right\|}

\title[On Finite Pseudorandom Binary Sequences]{On Finite Pseudorandom Binary
  Sequences: \\Generalized polynomials}

\author[M. G. Madritsch]{Manfred G. Madritsch}
\address[M. G. Madritsch]{
  Université de Lorraine, CNRS, IECL, F-54000 Nancy, France}
\email{manfred.madritsch@univ-lorraine.fr}

\author[R. F. Tichy]{Robert F. Tichy}
\address[R. F. Tichy]{
  Graz University of Technology, Institute of Analysis and Number
  Theory, 8010~Graz, Austria}
\email{tichy@tugraz.at}



\date{\today}

\begin{document}

\begin{abstract}
  In the present paper we generate binary pseudorandom sequences using
  generalized polynomials. A generalized polynomial is a function in whose
  description we not only allow addition and product (as it is the case in usual
  polynomials) but also the floor function. We estimate the well-distribution
  measure, looking at the ``randomness'' along arithmetic progressions.
\end{abstract}

\maketitle

\section{Introduction}

A finite pseudorandom binary sequence is a finite sequence over the alphabet
$\{-1,+1\}$. These sequences and their consideration were introduced amongst
others by Mauduit and Sárközy
\cite{mauduit_sarkoezy1997:finite_pseudorandom_binary} and play an important
role in financial mathematics as well as computer sciences. 

In order to quantify the random behaviour of such a sequence, Mauduit and
Sárközy introduced three measurements in
\cite{mauduit_sarkoezy1997:finite_pseudorandom_binary} : normality,
well-distribution along arithmetic progressions and small multiple correlations.
The present article focuses on the well-distribution measure -- the randomness
of the sequence by only looking at arithmetic progressions. A clever way to
investigate them is to look at sums of elements (along arithmetic progressions).
In particular, if the sum is rather large (in absolute value), then there must
be more elements of one kind, meaning that the sequence cannot be very
``random''. This simple observation is also the reason why we preferred the
alphabet $\{-1,+1\}$ over $\{0,1\}$.

Now we want to be more precise. Let
$E_N=\left(e_n\right)_{n=1}^N\in\{-1,+1\}^{N}$ be a finite pseudorandom binary
sequence. Then for three integers $a\in\NN^*$, $b\in\ZZ$ and $M\in\NN^*$ such
that $1\leq a+b<aM+b\leq N$, we consider a sum of the form
\[
  U(E_N,M,a,b)=\sum_{m=1}^M e_{am+b}.
\]
The well-distribution measure is the maximal absolute value we might obtain by
suitable choosing $a,b$ and $M$, \textit{i.e.}
\[
  W(E_N)=\max_{a,b,M}\abs{U(E_N,M,a,b)}=\max_{a,b,M}\abs{\sum_{m=1}^M e_{am+b}}.
\]
Clearly we have $W(E_N)\leq N$. The limiting distribution of the
well-distribution measure of a true random binary sequences has been established
by Aistleitner \cite{aistleitner2013:limit_distribution_well}.

A canonical way for constructing a pseudorandom binary sequence, is to take the
fractional part of a sequence of real numbers. In particular, we set
\[
  \chi(x)=\begin{cases}
    +1 &\text{if }x-\left\lfloor x\right\rfloor< \frac12,\\
    -1 &\text{if }x-\left\lfloor x\right\rfloor\geq \frac12.
  \end{cases}
\]
Then every finite sequence $(u_n)_{n=1}^N$ of real numbers gives rise to a
pseudorandom binary sequence via $e_n=\chi(u_n)$ for $1\leq n\leq N$. Moreover
this choice links the estimation of the well-distribution measure with the
discrepancy of the underlying sequence. We define the discrepancy $D_N$ of a
sequence $\left(u_n\right)_{n=1}^N$ of real numbers by
\[
  D_N\left(\left(u_n\right)_{n=1}^N\right)
  =\sup_{0\leq a<b\leq 1}\left|
    \frac{\#\{1\leq n\leq N\colon \{u_n\}\in[a,b[}{N} - (b-a)\right|.
\]
Then for $a\in\NN^*$, $b\in\ZZ$ and $M\in\NN^*$ such that $1\leq a+b\leq
aM+b\leq N$, we have 
\begin{multline*}
  \abs{U(E_N,M,a,b)}
  =\abs{\sum_{m=1}^M e_{am+b}}\\
  \leq \abs{\sum_{m=1}^M\left(1\!\!1_{\chi\left(u_{am+b}\right)<\tfrac12}-\frac12\right)}
  +\abs{\sum_{m=1}^M\left(1\!\!1_{\chi\left(u_{am+b}\right)\geq\tfrac12}-\frac12\right)}\leq 2D_M\left(\left(u_{am+b}\right)_{m=1}^M\right).
\end{multline*}
Therefore uniformly distributed sequences and sequences with low discrepancy are
good candidates for pseudorandom sequences with low well-distribution measure.
In the present paper we focus on polynomial-like sequences. For different
sequences and more general information on uniform distribution and the
discrepancy we refer the interested reader to the book of Drmota and Tichy
\cite{drmota_tichy1997:sequences_discrepancies_and}.

Mauduit and Sárközy considered in Part 5
\cite{mauduit_sarkoezy2000:finite_pseudorandom_binary} and Part 6
\cite{mauduit_sarkoezy2000:finite_pseudorandom_binary2} of their series the case
$f(n)=\alpha n^k$ with $\alpha$ real and $k\geq1$ an integer. Their results
depend on the coefficients of the continued fraction expansion of $\alpha$. If,
on the one hand, the coefficients are bounded, then they could show, that there
exists an $\eta$ depending only on $\alpha$ and $k$ such that
\[
  W(E_N)\ll N^{1-\eta+\varepsilon},
\]
where the implied constant also depends on $\alpha$ and $k$. If, on the other
hand, there is no such bound for the coefficients of the continued fraction
expansion of $\alpha$, then they could show that
\[
  W(E_N)\gg N.
\]

These considerations were extended to the Piatetski-Shapiro sequence $u_n=n^c$,
where $c>1$ is not an integer. Mauduit, Rivat and Sárközy
\cite{mauduit_rivat_sarkoezy2002:pseudo_random_properties} could show an
estimate $U(E_N,M,a,b)$, provided that $a$ is not too large (with respect to
$N$). These functions are part of the larger class of pseudo-polynomials. A
pseudo-polynomial is a function $f(x)$ of the form
\[
  f(x)=\alpha_dx^{\beta_d}+\cdots+\alpha_1x^{\beta_1},
\]
where $\alpha_1,\ldots,\alpha_d,\beta_1,\ldots,\beta_d$ are reals with
$\alpha_d\neq0$, $\beta_d>\cdots>\beta_1>0$ and at least one
$\beta_i\not\in\ZZ$.

The sequences of pseudo-polynomials carry many properties of polynomials like
forming a Poincaré set $P$. Let $(X,\mathfrak{B},\mu,T)$ be a measure preserving
dynamical system. Then we call a set $P\subset\NN^*$ a Poincaré set (or set of
recurrence) if for each set $A\in\mathfrak{B}$ of positive measure $\mu(A)>0$,
there exists $n\in P$ such that $A\cap T^n(A)\neq \emptyset$. Bergelson
\textit{et al.}
\cite{bergelson_kolesnik_madritsch+2014:uniform_distribution_prime} showed that
the sets formed by a pseudo-polynomial sequence are Poincaré sets. This was
later extended by the authors \cite{madritsch_tichy2016:dynamical_systems_and}
to other polynomial like families as well as to questions around approximation
in \cite{madritsch_tichy2019:multidimensional_van_der}. For more on the connection of uniform distribution, dynamical systems and
approximation we refer to the monograph of Montgomery
\cite{montgomery1994:ten_lectures_on}.

Inspired by these results, Bergelson, Kolesnik and Son
\cite{bergelson_kolesnik_son2019:uniform_distribution_subpolynomial} considered
the even larger class of Hardy fields. These fields consist of functions that
behave asymptotically like a power function, however, they are not polynomials.
Therefore even if they behave like an integer power their derivative does not
vanish identically, which is normally a mayor problem in the consideration. In
their paper \cite{bergelson_kolesnik_son2019:uniform_distribution_subpolynomial}
they showed, amongst other things, that the set $\{f(n)\colon n\in\NN^*\}$,
where $f$ is from a Hardy field is a Poincaré set. Using this more general class
of functions Rivat together with the two authors
\cite{madritsch_rivat_tichy2024:finite_pseudorandom_binary} extended the above
results for the well-distribution measure. Thereby they got rid of the
dependency on the size of $a$ providing a general bound.

Another class of polynomial-like function became of recent interest in dynamical
systems. This is the set of generalized polynomials $\mathrm{GP}$, where we also
allow the floor function in their description. To be more specific let
$\mathrm{GP}_0$ be the usual ring of polynomials $\RR[x]$ with real
coefficients. Then $\mathrm{GP}=\bigcup_{n=1}^{\infty} \mathrm{GP}_n$, where,
for $n\geq1$, we recursively define
\[
  \mathrm{GP}_n=\mathrm{GP}_{n-1}\cup
  \{v+w\colon v,w\in \mathrm{GP}_{n-1}\}\cup
  \{v\cdot w\colon v,w\in \mathrm{GP}_{n-1}\}\cup
  \{\left\lfloor v\right\rfloor\colon v\in \mathrm{GP}_{n-1}\}.
\]
Bergelson, H\aa land and Son
\cite{bergelson_ha_son2021:extension_weyls_equidistribution} recently
characterized that $u_n=f(n)$ with $f\in \mathrm{GP}$ is uniformly distributed
provided $g$ fulfils some canonical conditions, which are very similar to the
notion of finite type below. However, they do not give a general discrepancy
estimate. Hofer and Ramaré \cite{hofer_ramare2016:discrepancy_estimates_some}
and Mukhopadhyay, Ramaré and Viswanadham
\cite{mukhopadhyay_ramare_viswanadham2018:discrepancy_estimates_generalized}
considered the discrepancy of the sequence $u_n=f(n)$ where
$f(n)=\beta\left\lfloor \alpha p(n)\right\rfloor$ with $p$ a monic polynomial
and $\alpha$ and $\beta$ reals, such that $(\alpha,\alpha\beta)$ is of finite
type $t$.

\begin{defi}
  Let $\gamma_1,\ldots,\gamma_s$ be reals and $t>0$. Then we say that
  $(\gamma_1,\ldots,\gamma_s)$ is of finite type $t$ if there exists a constant
  $c(\varepsilon,\gamma_1,\ldots,\gamma_s)=c>0$ such that
  \[
    \prod_{j=1}^s\left(\max\left(1,\abs{n_j}\right)\right)^{t+\varepsilon}
    \left\|\sum_{j=1}^s n_j\gamma_j\right\|\geq c,
  \]
  where $\left\|\cdot\right\|$ denotes the distance to the nearest integer.
\end{defi}

The aim of the present paper is to extend our above considerations of the well-distribution measure for sequences based on function from Hardy fields to these general polynomials.

\begin{thm}\label{thm:main}
  Let $\alpha_1,\alpha_2,\beta$ be reals and let
  \[
    f(x)=\beta\left\lfloor \alpha_1\left\lfloor \alpha_2 p(x)\right\rfloor\right\rfloor,
  \]
  where $p\in\RR[x]$ is a monic polynomial of degree $d\geq2$. Furthermore we
  denote by $E_N=\left(e_n\right)_{n=1}^N$ with $e_n=\chi(f(n))$ the
  corresponding pseudorandom binary sequence. Suppose that
  \[
    (\alpha_2,\alpha_1\alpha_2,\alpha_1\alpha_2\beta)
  \]
  is of finite type $t$. Then there exists $\eta=\eta(f,t)>0$ depending only on
  $f$ and $t$ such that
  \[
    W(E_N)\ll N^{1-\eta+\varepsilon},
  \]
  where the implied constant depends on $f$, $t$ and $\varepsilon$.
\end{thm}

Note that we may calculate $\eta$ explicitly for a given function $f$ and finite
type $t$. However, since this exponent is far from the expected one, we omit the
details to the reader.

\section{Utilities}

\subsection{Discrepancy estimates}

As we already mentioned above the well-distribution measure is in strong
relation with the discrepancy of a sequence modulo $1$. A standard way of
estimating is the Erd\H{o}s-Turán inequality.
\begin{lem}[\textsc{Erd\H{o}s-Turán}]
  \label{lem:erdos-turan}
  For any integers $N>0$, $H>0$, and any sequence $(u_n)_{n=1}^N$ of $N$ real numbers we have
  \[
    D_N\left((u_n)_{n=1}^N\right)\leq\frac2{H+1}
      + 2\sum_{h=1}^H\frac1h\abs{\frac1N\sum_{n=1}^N e(hu_n)}.
  \]
\end{lem}

\begin{proof}
  For a proof see Lemma 2.5 of Kuipers and Niederreiter
  \cite{kuipers_niederreiter1974:uniform_distribution_sequences} (for smaller
  coefficients see also Rivat and Tenenbaum
  \cite{rivat_tenenbaum2005:constantes_derd_os}).
\end{proof}

For sequences of the form $u_n=n\theta$ with real $\theta$ we the following
estimate will be more useful.
\begin{lem}
  \label{lem:discrepancy_nearest_integer}
  Let $\theta\in\RR\setminus\QQ$ be an irrational number. Then for any positive integers $L\geq1$ and $J\geq1$ the discrepancy of the sequence $(\ell\theta)_{\ell=1}^{L}$ satisfies
  \[
    D_L\left(\left(\ell\theta\right)_{\ell=1}^L\right)
    \leq C\left(\frac{1}{J}+\frac{1}{L}\sum_{j=1}^J\frac{1}{j\norm{j\theta}}\right),
  \]
  where $C$ is an absolute constant.
\end{lem}

\begin{proof}
  This is \cite[Lemma 3.2]{kuipers_niederreiter1974:uniform_distribution_sequences}.
\end{proof}

\subsection{Floor function}

In the description of general polynomials, the floor function is allowed.
However, we will rewrite the floor function using the fractional part since this
is a $1$-periodic function. For reals $x,\tau\in\RR$ we write
\begin{gather}\label{eq:F_definition}
  F(x,\tau)=e\left(\tau\{x\}\right).
\end{gather}
Since the fractional part is discontinuous at the integers, we use convolution
for smoothing and introduce 
\begin{gather}\label{eq:Gr_definition}
  G_r(x,\tau,\delta)=\frac{1}{(2\delta)^r}
  \left(1\!\!1_{[-\delta,\delta]} \star \cdots \star 1\!\!1_{[-\delta,\delta]}
  \star F(x,\tau)\right),
\end{gather}
where $r\geq1$ is an integer and $\delta>0$. The following lemma deals with the
error we have in considering $G$ instead of $F$.
\begin{lem}
  \label{lem:approximation_by_G}
  For any sequence $\{u_n\}_{n\geq0}$ of real numbers, and any positive integer $N$ we have
  \[
    \sum_{0\leq n<N}\abs{F(u_n,\tau)-G_r(u_n,\tau,\delta)}
    \ll Nr\delta+Nr^2\delta\abs{\tau}+ND_{N}(u_n).
    \]
\end{lem}

\begin{proof}
  This is Lemma 10 of Hofer and Ramaré
  \cite{hofer_ramare2016:discrepancy_estimates_some}.
\end{proof}

Now, as usual, we denote by $\widehat{G_r}$ the (discrete) Fourier transform of $G_r$ (with respect to $x$):
\[
  G_r(x,\tau,\delta)=\sum_{k\in\ZZ} \widehat{G_r}(k,\tau,\delta) e\left(-kx\right).
\]
The properties of this transformation are twofold. On the one hand we use the
following lemma to truncate the infinite sum.

\begin{lem}\label{lem:large_Fourier_coefficients}
  Let $K$ be a positive integer such that $\abs{\tau+k}\geq\frac{k}{2}$ for $k\in\ZZ$ with $\abs{k}>K$. Then
  \[
    \sum_{\abs{k}>K}\widehat{G_r}(k,\tau,\delta)\ll\left(\delta K\right)^{-r}.
  \]
\end{lem}

\begin{proof}
  This is Lemma 3 of Mukhopadhyay \textit{et al.} \cite{mukhopadhyay_ramare_viswanadham2018:discrepancy_estimates_generalized}.
\end{proof}

On the other hand the second lemma considers the case of higher moments of the
Fourier coefficients.

\begin{lem}\label{lem:p-norm_Fourier_coefficients}
  Let $\tau\in\RR$ and $0<\delta<\min\left(\frac{1}{2\abs{\tau}},1\right)$. Then for $p>1$ we have
  \[
    \sum_{k\in\ZZ}\abs{\widehat{G_r}(k,\tau,\delta)}^p \ll_p 1.
  \]
\end{lem}

\begin{proof}
  This is Lemma 4 of Mukhopadhyay \textit{et al.} \cite{mukhopadhyay_ramare_viswanadham2018:discrepancy_estimates_generalized}.
\end{proof}

\subsection{Exponential sum estimates}

Above we stated the Erd\H{o}s-Turán inequality (Lemma \ref{lem:erdos-turan})
turning the estimation of the discrepancy into one of exponential sums. The aim
here is Lemma~\ref{lem:type_t_exponential_sum} below. However, we need some
tools for its proof. The first deals with the classical idea of Weyl differencing. 

\begin{lem}\label{lem:weyl_differencing}
  Suppose that $\lambda_1,\lambda_2,\ldots,\lambda_N$ is a sequence of complex numbers, each with $\abs{\lambda_i}\leq 1$, and define $\Delta \lambda_m=\lambda_m$, $\Delta_r\lambda_m=\lambda_{m+r}\overline{\lambda_m}$ and
  \[
    \Delta_{r_1,\ldots,r_k,s}\lambda_m=\left(\Delta_{r_1,\ldots,r_k}\lambda_{m+s}\right)\overline{\left(\Delta_{r_1,\ldots,r_k}\lambda_m\right)}.
  \]
  Then for any given $k\geq1$, and real number $Q\in[1,N]$,
  \[
    \abs{\frac{1}{8N}\sum_{m=1}^N\lambda_m}^{2^k}
    \leq \frac{1}{8Q}+\frac{1}{8Q^{2-2^{-k+1}}}
    \sum_{r_1=1}^Q\sum_{r_2=1}^{Q^{\frac12}}\cdots\sum_{r_k=1}^{Q^{2^{-k+1}}}\abs{\frac{1}{N}\sum_{m=1}^{N-r_1-\cdots-r_k}
    \Delta_{r_1,\ldots,r_k}\lambda_m}.
  \]
\end{lem}

\begin{proof}
  This is a variant of a lemma of Weyl-van der Corput (see \cite[Lemma 2.7]{graham_kolesnik1991:van_der_corputs}) as given in \cite[Lemma 8.3]{granville_ramare1996:explicit_bounds_on}.
\end{proof}

Weyl differencing turns the sum into a linear one. The second tool connects this
linear exponential sum with approximation properties of the coefficient.

\begin{lem}
  \label{lem:linear_exponential_sum}
  For every real number $\alpha$ and all integers $N_1<N_2$,
  \[\sum_{n=N_1+1}^{N_2}e(\alpha n)
  \ll \min\left(N_2-N_1,\lVert\alpha\rVert^{-1}\right).\]
\end{lem}

\begin{proof}
  This is Lemma 4.7 of Nathanson \cite{nathanson1996:additive_number_theory}.
\end{proof}

Finally we need a link between the sum of the minima and the discrepancy of $n\alpha$-sequences.

\begin{lem}
  \label{lem:sum_of_minima}
  Let $\xi\in\RR$ and let $L\in\NN^*$ be a positive integer. Then
  \[
    \sum_{\ell=1}^L\min\left(N,\left\|\ell\xi\right\|^{-1}\right)
    \ll L\log N\left(1+N D_L\left(\left(\ell\xi\right)_{\ell=1}^L\right)\right).
  \]
\end{lem}

\begin{proof}
  We divide the interval $[0,1]$ into $N$ parts and denote by $E_n$ the number
  of elements $\norm{\ell \xi}$ that fall in the $n$-th, \textit{i.e.} for
  $0\leq n<N$ we set
  \[
    E_n:=\#\left\{ \ell\leq L\colon \frac{n}{N}<\norm{\ell\xi}\leq\frac{n+1}{N}\right\}.
  \]
  Clearly we have
  \begin{align*}
    \sum_{\ell=1}^{L}\min\left(N,\frac{1}{\norm{\ell\xi}}\right)
    &=NE_0+\sum_{\ell\not\in E_0}\frac{1}{\norm{\ell\xi}}
    \leq NE_0+\sum_{n=1}^{N-1}\frac{N}{n}E_n.
  \end{align*}
  
  By the definition of discrepancy we obtain that
  \[
    E_n=\frac{2L}{N}+\mathcal{O}\left(L D_L\left(\left(\ell\xi\right)_{\ell=1}^L\right)\right).
  \]
  Plugging this into the sum of the minima yields the desired bound.
\end{proof}

\section{A discrepancy estimate}

We start our considerations with a first discrepancy estimate.

\begin{prop}
  \label{prop:discrepancy_basis_case}
  Let $p\in\RR[X]$ be a monic polynomial of degree $d\geq2$ and let $a\in\NN^*$
  and $b\in\ZZ$ be integers. Furthermore suppose that $\alpha$ is of finite type
  $t$ and that
  \[
    a\leq N^{\frac{2-2^{2-d}}{dt}}.
  \]
  Then for $\varepsilon>0$ 
  \[
    D_N\left(\left(\alpha p(an+b)\right)_{n=1}^N\right)
    \ll a^{\frac{dt}{2^{d-1}(t+1)+t}+\varepsilon}N^{-\frac{2-2^{2-d}}{2^{d-1}(2t+1)+t}+\varepsilon},
  \]
  where the implied constant depends on $\varepsilon$, $d$ and $t$.
\end{prop}

Although this seems very classic, because no floor function is involved in this
variant, we want to provide a proof here for the sake of completeness. This also
allows us to present and prove our general tool for the occurring exponential
sums.

\begin{lem}\label{lem:type_t_exponential_sum}
  Let $p$ be a monic polynomial of degree $d\geq2$ and $a\in\NN^*$ and $b\in\ZZ$
  be two integers. Furthermore suppose that $(\gamma_1,\ldots,\gamma_s)\in\RR^s$
  is of finite type $t$ with $t>0$.
  Then for $N$ sufficiently large we have
  \[
    \abs{\sum_{n=1}^N e\left(\left(k_1\gamma_1+\cdots+k_s\gamma_s\right)
    p(an+b)\right)}^{2^{d-1}}
    \ll
    \left(a^{sd}\abs{k_1\cdots k_s}\right)^{\frac{t}{st+1}+\varepsilon}
    N^{2^{d-1}-\frac{2-2^{2-d}}{st+1}+\varepsilon}.
  \]
\end{lem}

\begin{proof}
  We start by an application of Lemma \ref{lem:weyl_differencing} with $Q=N$ to get
  \begin{multline*}
    \abs{\sum_{n=1}^N e\left(\left(k_1\gamma_1+\cdots+k_s\gamma_s\right)
    p(an+b)\right)}^{2^{d-1}}\\
    \ll N^{2^{d-1}-1}
    +N^{2^{d-1}+2^{2-d}-3}
    \sum_{r_1=1}^{N} \cdots \sum_{r_{d-1}=1}^{N^{2^{2-d}}}
    \abs{\sum_{n=1}^N
    e\left(d!\left(k_1\gamma_1+\cdots+k_s\gamma_s\right)a^d r_1\cdots r_{d-1}n\right) }.
  \end{multline*}

  Applying Lemma \ref{lem:linear_exponential_sum} to the innermost sum yields
  \begin{multline*}
    \abs{\sum_{n=1}^N e\left(\left(k_1\gamma_1+\cdots+k_s\gamma_s\right)
    p(an+b)\right)}^{2^{d-1}}\\
    \ll N^{2^{d-1}-1}
    +N^{2^{d-1}+2^{2-d}-3}
    \sum_{r_1=1}^{N} \cdots \sum_{r_{d-1}=1}^{N^{2^{2-d}}}
    \min\left(N,\norm{d!\left(k_1\gamma_1+\cdots+k_s\gamma_s\right)a^d r_1\cdots r_{d-1}}^{-1}\right).
  \end{multline*}

  We want to get rid of the multiple sum. Therefore we denote by $T(m)$ the number of representations of $m$ as product of $r_1\cdots r_{d-1}$, \textit{i.e.}
  \[
    T(m)=
    \abs{\left\{(r_1,\ldots,r_{d-1})\in
      [1,N]\times\cdots\times[1,N^{2^{2-d}}]\colon
      m=r_1\cdots r_{d-1}\right\}}.
  \]
  Thus
  \begin{multline*}
    \abs{\sum_{n=1}^N e\left(\left(k_1\gamma_1+\cdots+k_s\gamma_s\right)
    p(an+b)\right)}^{2^{d-1}}\\
    \ll N^{2^{d-1}-1}
    +N^{2^{d-1}+2^{2-d}-3}
    \sum_{m=1}^{N^{2-2^{2-d}}}T(m)
    \min\left(N,\norm{d!\left(k_1\gamma_1+\cdots+k_s\gamma_s\right)a^d m}^{-1}\right).
  \end{multline*}

  Using estimates for the divisor function (\textit{cf.} the proof of
  \cite[Lemma 4.14]{nathanson1996:additive_number_theory}) we get that $T(m)\ll
  m^{\varepsilon}\ll N^{\varepsilon}$ and therefore
  \begin{multline*}
    \abs{\sum_{n=1}^N e\left(\left(k_1\gamma_1+\cdots+k_s\gamma_s\right)
    p(an+b)\right)}^{2^{d-1}}\\
    \ll N^{2^{d-1}-1}
    +N^{2^{d-1}+2^{2-d}-3+\varepsilon}
    \sum_{m=1}^{N^{2-2^{2-d}}}\min\left(N,\norm{d!(k_1\gamma_1+\cdots+k_s\gamma_s)a^dm}^{-1}\right).
  \end{multline*}

  Now by Lemma \ref{lem:sum_of_minima} we get for the sum of minima that
  \begin{multline}\label{eq:9}
    \sum_{m=1}^{L}\min\left(N,\norm{d!(k_1\gamma_1+\cdots+k_s\gamma_s)a^dm}^{-1}\right)\\
    \ll L(\log N)\left(1+ND_L\left(\left(d!(k_1\gamma_1+\cdots+k_s\gamma_s)a^dm\right)_{m=1}^L\right)\right),
  \end{multline}
  where we have set $L=N^{2-2^{2-d}}$ for short.

  Since $(\gamma_1,\ldots,\gamma_s)$ is of finite type $t$, it exists $c=c(\varepsilon,\gamma_1,\ldots,\gamma_s)>0$ such that
  \[
    \norm{d!(k_1\gamma_1+\cdots+k_s\gamma_s)a^dm}
    \geq \frac{c(\varepsilon,\gamma_1,\ldots,\gamma_s)}
      {(d!a^dm)^{st+\varepsilon}\abs{k_1\cdots k_s}^{t+\varepsilon}}.
  \]
  Thus by Lemma \ref{lem:discrepancy_nearest_integer} we get for $J\geq1$ that
  \begin{align*}
    D_L\left(\left(d!(k_1\gamma_1+\cdots+k_s\gamma_s)a^dm\right)_{m=1}^L\right)
    &\ll \frac{1}{J}+\frac{1}{L}\sum_{j=1}^J j^{-1}(a^dj)^{st+\varepsilon}\abs{k_1\cdots k_s}^{t+\varepsilon}\\
    &\ll \frac{1}{J}+\frac{1}{L}(a^dJ)^{st+\varepsilon}\abs{k_1\cdots k_s}^{t+\varepsilon}.
  \end{align*}
  
  Choosing
  \[
    J= a^{-d\frac{st}{st+1}}\abs{k_1\cdots k_s}^{-\frac{t}{st+1}}L^{\frac{1}{st+1}}
  \]
  we get
  \[
    \sum_{m=1}^{L}\min\left(N,\left\|d!ma^d(h\alpha_2-k)\alpha\right\|^{-1}\right)
    \ll a^{d\frac{st}{st+1}+\varepsilon}
    \abs{k_1\cdots k_s}^{\frac{t}{st+1}+\varepsilon}
    N^{3-2^{2-d}-\frac{2-2^{2-d}}{st+1}+\varepsilon},
  \]
  which together with \ref{eq:9} proves the lemma.
\end{proof}



With this tool in hand we are able to prove the first estimation.

\begin{proof}[Proof of Proposition \ref{prop:discrepancy_basis_case}]
  The proof follows three steps. First, we use the Erd\H{o}s-Turán inequality
  which transforms the discrepancy estimate in an exponential sum estimate. Then
  we use Lemma~\ref{lem:type_t_exponential_sum}, our key lemma for the
  exponential sums. Finally we choose the parameter $H$ in the Erd\H{o}s-Turán
  inequality accordingly. This last step provides us with the bound on $a$.

  Starting with an application of the Erd\H{o}s-Turán inequality (Lemma \ref{lem:erdos-turan}) we obtain
  \begin{gather}\label{eq:3}
    D_N\left(\left(\alpha p(an+b)\right)_{n=1}^N\right)
    \ll \frac1{H}
    + \sum_{h=1}^H\frac1h\abs{\frac1N\sum_{n=1}^N e\left(h\alpha p(an+b)\right)},
  \end{gather}
  where $H\geq1$ is a parameter we will choose later.
  
  An application of Lemma \ref{lem:type_t_exponential_sum} yields
  \begin{gather*}
    \abs{\sum_{n=1}^Ne(h\alpha p(an+b))}^{2^{d-1}}
    \ll N^{2^{d-1}-\frac{2-2^{2-d}}{t+1}+\varepsilon}
      \left(a^dh\right)^{\frac{t}{t+1}+\varepsilon}.
  \end{gather*}

  Finally plugging everything into \eqref{eq:3} and setting
  \[
    H=\left\lceil N^{\frac{2-2^{2-d}}{2^{d-1}(2t+1)}}a^{\frac{dt}{2^{d-1}(2t+1)+t}}\right\rceil\geq1
  \]
  proves the proposition.
\end{proof}

\section{The first iteration}

Now we add the floor function to our generating function and establish the following proposition.

\begin{prop}\label{prop:discrepancy_first_iteration}
  Let $\alpha$ and $\beta$ be reals and let $p\in\RR[x]$ be a monic polynomial of degree $d\geq2$. Suppose that $(\alpha,\alpha\beta)$ is of finite type $t$ with $t>0$.
  Then
  \begin{align*}
    D_N\left(\left(\beta\left\lfloor\alpha p(an+b)\right\rfloor\right)_{n=1}^N\right)
    &\ll a^{\frac{2dt}{2^{d-1}(2t+1)+4t+1}}
    N^{-\frac{2-2^{2-d}}{2^{d-1}(2t+1)+7t+2}}.
  \end{align*}
\end{prop}

\begin{proof}
  By Lemma \ref{lem:erdos-turan} we get that
  \begin{equation}\label{eq:8}
      D_N\left(\left(\beta\left\lfloor\alpha p(an+b)\right\rfloor\right)_{n=1}^N\right)
      \ll \frac1{H}
      + \sum_{h=1}^H\frac1h\abs{\frac1N\sum_{n=1}^N e\left(h\beta\left\lfloor\alpha p(an+b)\right\rfloor\right)},
  \end{equation}
  where $H\geq1$ is a positive integer we choose later.

  As above we concentrate on the exponential sum. Using our definitions of $F$ and
  $G_r$ with integer $r\geq1$ and real $\delta>0$ in \eqref{eq:F_definition} and
  \eqref{eq:Gr_definition}, respectively, we obtain
  \begin{equation}\label{eq:7}
    \begin{split}
      \sum_{n=1}^N e\left(h\beta\left\lfloor\alpha p(an+b)\right\rfloor\right)
      &=\sum_{n=1}^N e\left(h\beta\alpha p(an+b)\right) F\left(\alpha p(an+b),-h\beta\right)\\
      &=\sum_{n=1}^N e\left(h\beta\alpha p(an+b)\right) G_r\left(\alpha p(an+b),-h\beta,\delta\right)
      +\mathcal{O}\left(R\right),
    \end{split}
  \end{equation}
  where
  \[
    R=\sum_{n=1}^N\abs{F(\alpha p(an+b),-h\beta)-G_r(\alpha p(an+b),-h\beta,\delta)}.
  \]

  Using Lemma \ref{lem:approximation_by_G} together with the estimate of the
  distribution in Proposition \ref{prop:discrepancy_basis_case} we get for the
  error term $R$ that
  \begin{align*}
    R
    &\ll Nr\delta+Nr^2\delta\abs{h\beta}+ND_N\left(\left(\alpha p(an+b)\right)_{n=1}^N\right)\\
    &\ll Nr\delta+Nr^2\delta\abs{h\beta}
      +a^{\frac{dt}{2t+1}+\varepsilon}N^{1-\frac{2-2^{2-d}}{2^{d-1}(2t+1)}+\varepsilon}.
  \end{align*}

  We return to the exponential sum. Fourier transforming $G_r$ we get
  \begin{equation}\label{eq:6}
    \begin{split}
      &\sum_{n=1}^Ne\left(h\beta\alpha p(an+b)\right)
        G_r\left(\alpha p(an+b),-h\beta,\delta\right)\\
      &\quad=\sum_{k\in\ZZ} \widehat{G_r}(k,-h\beta,\delta) \sum_{n=1}^Ne\left((h\beta-k)\alpha p(an+b)\right)\\
      &\quad=\sum_{\abs{k}\leq K} \widehat{G_r}(k,-h\beta,\delta) \sum_{n=1}^Ne\left((h\beta-k)\alpha p(an+b)\right)
      +\mathcal{O}\left(N(\delta K)^{-r}\right),
    \end{split}
  \end{equation}
  where we used Lemma \ref{lem:large_Fourier_coefficients} in the last step.

  Again we concentrate on the weighted exponential sum. The idea is to apply
  Lemma \ref{lem:type_t_exponential_sum}. Using Holder's inequality we separate
  the Fourier coefficient and the exponential sum. Thus
  \begin{align*}
    &\sum_{\abs{k}\leq K} \widehat{G_r}(k,-h\beta,\delta) \sum_{n=1}^Ne\left((h\beta-k)\alpha p(an+b)\right)\\
    &\quad\ll \left(\sum_{\abs{k}\leq K}\abs{\widehat{G_r}(k,-h\beta,\delta)}^{\frac{2^{d-1}}{2^{d-1}-1}}\right)^{\frac{2^{d-1}-1}{2^{d-1}}}\left(\sum_{\abs{k}\leq K}\abs{\sum_{n=1}^N e\left((h\beta-k)\alpha p(an+b)\right)}^{2^{d-1}}\right)^{\frac{1}{2^{d-1}}}\\
    &\quad\ll \left(\sum_{\abs{k}\leq K}\abs{\sum_{n=1}^N e\left((h\beta-k)\alpha p(an+b)\right)}^{2^{d-1}}\right)^{\frac{1}{2^{d-1}}},
  \end{align*}
  where, this time, we used Lemma \ref{lem:p-norm_Fourier_coefficients} in the last step.

  Since the vector $(\alpha,\alpha\beta)$ is of finite type $t$, an application
  of Lemma \ref{lem:type_t_exponential_sum} yields
  \[
    \abs{\sum_{n=1}^Ne((h\beta-k)\alpha p(an+b))}^{2^{d-1}}\\
    \ll a^{d\frac{2t}{2t+1}+\varepsilon}
    \abs{hk}^{\frac{t}{2t+1}+\varepsilon}
    N^{2^{d-1}-\frac{2-2^{2-d}}{2t+1}+\varepsilon}.
  \]
  Summing over $k$ we get
  \begin{multline*}
    \sum_{\abs{k}\leq K} \widehat{G_r}(k,-h\beta,\delta) \sum_{n=1}^Ne\left((h\beta-k)\alpha p(an+b)\right)\\
    \quad\ll a^{\frac{d}{2^{d-1}}\frac{2t}{2t+1}+\varepsilon}
    \abs{h}^{\frac{t}{2^{d-1}(2t+1)}+\varepsilon}
    K^{\frac{3t+1}{2^{d-1}(2t+1)}+\varepsilon}
    N^{1-\frac{2-2^{2-d}}{2^{d-1}(2t+1)}+\varepsilon}.
  \end{multline*}

  Let $\rho>0$ and $\theta>0$ be real parameters, which we will choose in an
  instant. Then we set
  \[
    \delta^{-1}=hN^{\theta}
    \quad\text{and}\quad
    K=h^\rho N^{\theta},
  \]
  Plugging every estimate we have so far in \eqref{eq:6} and then in \eqref{eq:7} and \eqref{eq:8} yields
  \begin{align*}
    D_N\left(\left(\beta\left\lfloor\alpha p(an+b)\right\rfloor\right)_{n=1}^N\right)
    &\ll a^{\frac{d}{2^{d-1}}\frac{2t}{2t+1}+\varepsilon}
    H^{\frac{t}{2^{d-1}(2t+1)}+\rho\frac{3t+1}{2^{d-1}(2t+1)}+\varepsilon}
    N^{-\frac{2-2^{2-d}}{2^{d-1}(2t+1)}+\theta\frac{3t+1}{2^{d-1}(2t+1)}+\varepsilon}\\
    &\quad+H^{r(1-\rho)}
    + H^{-1}+N^{-\theta}+N^{-\theta}\log H
    +a^{d\frac{t}{2t+1}+\varepsilon}N^{-\frac{2-2^{2-d}}{2^{d-1}(2t+1)}+\varepsilon}.
  \end{align*}

  Now we choose $\rho=1+\varepsilon_1$, with
  $\varepsilon_1=\varepsilon_1(\varepsilon,t)>0$ sufficiently small, and $r$ an
  integer such that $r>\frac{1}{\varepsilon_1}$. Then clearly $H^{r(1-\rho)}\ll
  H^{-1}$. Finally we set $H=a^{-\sigma}N^{\theta}$ with
  \[
    \sigma=\frac{2dt}{2^{d-1}(2t+1)+4t+1}
    \quad\text{and}\quad
    \theta=\frac{2-2^{2-d}}{2^{d-1}(2t+1)+7t+2}
  \]
  yielding the desired estimate.
\end{proof}

\section{The second iteration}

Now we iterate the process one step further and provide the following.

\begin{prop}\label{prop:discrepancy_second_iteration} Let $\alpha_1$, $\alpha_2$
  and $\beta$ be reals and let $p\in\RR[x]$ be a monic polynomial of degree
  $d\geq2$. Suppose that the vector $(\alpha_2, \alpha_1\alpha_2,
  \alpha_1\alpha_2\beta)$ is of finite type $t$ with $t>0$. Then
  \[
    D_N\left(\left(\beta\left\lfloor\alpha_1\left\lfloor\alpha_2 p(an+b)
      \right\rfloor\right\rfloor\right)_{n=1}^N\right)
    \ll
    a^{\frac{3dt}{2^{d}(3t+1)+5t+1}}
    N^{-\frac{(2-2^{2-d})2^{d-1}(3t+1)}{(2^{d-1}(3t+1)+21t+5)(2^{d}(3t+1)+4t+1)}}.
  \]
\end{prop}

\begin{proof}
  Again we start with the Erd\H{o}s-Turán inequality: For $H\geq1$ we obtain
  \begin{gather}\label{eq:13}
    D_N\left(\left(\beta\left\lfloor\alpha_1\left\lfloor\alpha_2 p(an+b)
    \right\rfloor\right\rfloor\right)_{n=1}^N\right)
    \ll
    \frac1H+
    \sum_{h=1}^H\frac1h\abs{\frac1N\sum_{n=1}^n
    e\left(h\beta\left\lfloor\alpha_1\left\lfloor\alpha_2 p(an+b)
    \right\rfloor\right\rfloor\right)}.
  \end{gather}
  
  Now we focus on the innermost exponential sum and use the approximation of the
  fractional part as above to get for an integer $r_1\geq1$ and a real
  $\delta_1>0$ (which we will fix later) that
  \begin{equation}\label{eq:10}
    \begin{split}
      &\sum_{n=1}^n
      e\left(h\beta\left\lfloor\alpha_1\left\lfloor\alpha_2 p(an+b)
      \right\rfloor\right\rfloor\right)\\
      &\quad=\sum_{n=1}^n
      e\left(h\beta\alpha_1\left\lfloor\alpha_2 p(an+b)
      \right\rfloor\right)F(\alpha_1\left\lfloor\alpha_2 p(an+b)
      \right\rfloor,-h\beta)\\
      &\quad=\sum_{n=1}^n
      e\left(h\beta\alpha_1\left\lfloor\alpha_2 p(an+b)
      \right\rfloor\right)G_{r_1}(\alpha_1\left\lfloor\alpha_2 p(an+b)
      \right\rfloor,-h\beta,\delta_1)+\mathcal{O}(R_1),
    \end{split}
  \end{equation}
  where
  \begin{gather}\label{eq:R_1}
    R_1=\sum_{n=1}^N
      \abs{F(\alpha_1\left\lfloor\alpha_2 p(an+b)\right\rfloor,-h\beta)
      - G_{r_1}(\alpha_1\left\lfloor\alpha_2 p(an+b)\right\rfloor,-h\beta,\delta_1)}.
  \end{gather}

  Again we use the Fourier series expansion of $G_{r_1}$ to get
  \begin{align*}
    &\sum_{n=1}^n
    e\left(h\beta\alpha_1\left\lfloor\alpha_2 p(an+b)
    \right\rfloor\right)G_{r_1}(\alpha_1\left\lfloor\alpha_2 p(an+b)
    \right\rfloor,-h\beta,\delta_1)\\
    &\quad=\sum_{k_1\in\ZZ}
    \widehat{G_{r_1}}(k_1,-h\beta,\delta_1)
    \sum_{n=1}^n
    e\left(\left(h\beta-k_1\right)\alpha_1\left\lfloor\alpha_2 p(an+b)
    \right\rfloor\right)\\
    &\quad=\sum_{\abs{k_1}\leq K_1}
    \widehat{G_{r_1}}(k_1,-h\beta,\delta_1)
    \sum_{n=1}^n
    e\left(\left(h\beta-k_1\right)\alpha_1\left\lfloor\alpha_2 p(an+b)
    \right\rfloor\right)
    +\mathcal{O}\left(N(\delta_1K_1)^{-r_1}\right),
  \end{align*}
  where we applied Lemma \ref{lem:large_Fourier_coefficients} in the last step.

  Since we got rid of only one floor function we need to iterate the whole
  procedure to eliminate the second one. By a similar argument we get
  \begin{multline*}
    \sum_{n=1}^n
    e\left(\left(h\beta-k_1\right)\alpha_1\left\lfloor\alpha_2 p(an+b)
    \right\rfloor\right)\\
    =
    \sum_{n=1}^n
    e\left(\left(h\beta-k_1\right)\alpha_1\left\lfloor\alpha_2 p(an+b)
    \right\rfloor\right)G_{r_2}\left(\alpha_2p(an+b),-(h\beta-k_1)\alpha_1,\delta_2\right)
    +\mathcal{O}\left(R_2\right),
  \end{multline*}
  where
  \begin{gather}\label{eq:R_2}
    R_2=\sum_{n=1}^N\abs{
      F\left(\alpha_2p(an+b),-(h\beta-k_1)\alpha_1\right)
      - G_{r_2}\left(\alpha_2p(an+b),-(h\beta-k_1)\alpha_1,\delta_2\right)
    }.    
  \end{gather}

  Furthermore using the Fourier series of $G_{r_2}$ and a truncation yields
  \begin{align*}
    &\sum_{n=1}^n
    e\left(\left(h\beta-k_1\right)\alpha_1\left\lfloor\alpha_2 p(an+b)
    \right\rfloor\right)G_{r_2}\left(\alpha_2p(an+b),-(h\beta-k_1)\alpha_1,\delta_2\right)\\
    &\quad=\sum_{\abs{k_2}\leq K_2}
    \widehat{G_{r_2}}\left(k_2,-(h\beta-k_1)\alpha_1,\delta_2\right)
    \sum_{n=1}^n
    e\left(\left(h\alpha_1\alpha_2\beta-k_1\alpha_1\alpha_2-k_2\alpha_2\right)p(an+b)
    \right)\\
    &\quad\quad+\mathcal{O}\left(N\left(\delta_2K_2\right)^{-r_2}\right).
  \end{align*}

  Let's pause for a moment and summarise what we have so far. Plugging
  everything into the original exponential sum in \eqref{eq:10} we have
  \begin{align}\label{eq:12}
    \abs{\frac1N\sum_{n=1}^n
      e\left(h\beta\left\lfloor\alpha_1\left\lfloor\alpha_2 p(an+b)
      \right\rfloor\right\rfloor\right)}
    \ll S_1
      +S_2
      +S_3
      +S_4
      +S_5,
  \end{align}
  where
  \begin{align*}
    S_1&=\sum_{\abs{k_1}\leq K_1}
    \widehat{G_{r_1}}(k_1,-h\beta,\delta_1)
    \sum_{\abs{k_2}\leq K_2}
    \widehat{G_{r_2}}\left(k_2,-(h\beta-k_1)\alpha_1,\delta_2\right)\\
    &\quad\times
    \abs{\frac1N\sum_{n=1}^n
    e\left(\left(h\alpha_1\alpha_2\beta-k_1\alpha_1\alpha_2-k_2\alpha_2\right)p(an+b)
    \right)}\\
    S_2&=\sum_{\abs{k_1}\leq K_1}
    \widehat{G_{r_1}}(k_1,-h\beta,\delta_1)\left(\delta_2K_2\right)^{-r_2}\\
    S_3&=\sum_{\abs{k_1}\leq K_1}
    \widehat{G_{r_1}}(k_1,-h\beta,\delta_1)\frac{R_2}{N}\\
    S_4&=\left(\delta_1K_1\right)^{-r_1} & &\text{and}\\
    S_5&=\frac{R_1}{N}.
  \end{align*}

  We start with $S_1$ since it is the only part involving sums over $k_1$ and $k_2$. Applying Hölder's inequality we get
  \begin{align*}
    \abs{S_1}
    &\ll
    \left(\sum_{\abs{k_1}\leq K_1}
    \sum_{\abs{k_2}\leq K_2}
    \left|\widehat{G_{r_1}}(k_1,-h\beta,\delta_1)
    \widehat{G_{r_2}}\left(k_2,-(h\beta-k_1)\alpha_1,\delta_2\right)\right|^{\frac{2^{d-1}}{2^{d-1}-1}}\right)^{\frac{2^{d-1}-1}{2^{d-1}}}\\
    &\quad\times
      \left(\sum_{\abs{k_1}\leq K_1}
      \sum_{\abs{k_2}\leq K_2}\left|
      \frac1N\sum_{n=1}^n
      e\left(\left(h\alpha_1\alpha_2\beta-k_1\alpha_1\alpha_2-k_2\alpha_2\right)p(an+b)
      \right)\right|^{2^{d-1}}\right)^{\frac{1}{2^{d-1}}}.
  \end{align*}
  
  Since $\frac{2^{d-1}}{2^{d-1}-1}>1$ a double application of Lemma
  \ref{lem:p-norm_Fourier_coefficients} yields
  \begin{align*}
    \abs{S_1}
    &\ll
      \left(\sum_{\abs{k_1}\leq K_1}
      \sum_{\abs{k_2}\leq K_2}\left|
      \frac1N\sum_{n=1}^n
      e\left(\left(h\alpha_1\alpha_2\beta-k_1\alpha_1\alpha_2-k_2\alpha_2\right)p(an+b)
      \right)\right|^{2^{d-1}}\right)^{\frac{1}{2^{d-1}}}.
  \end{align*}

  By assumption $(\alpha_1\alpha_2\beta, \alpha_1\alpha_2, \alpha_2)$ is of
  type $t$ and we may apply Lemma \ref{lem:type_t_exponential_sum} to get
  \begin{align*}
    \abs{S_1}
    &\ll
      \left(\sum_{\abs{k_1}\leq K_1}
      \sum_{\abs{k_2}\leq K_2}
      a^{\frac{3dt}{3t+1}+\varepsilon}
      \abs{hk_1k_2}^{\frac{t}{3t+1}+\varepsilon}
      N^{-\frac{2-2^{2-d}}{3t+1}+\varepsilon}
      \right)^{\frac{1}{2^{d-1}}}. 
  \end{align*}
  Summing over $k_2$ yields
  \begin{gather}\label{eq:11}
    \abs{S_1}
    \ll
      \left(\sum_{\abs{k_1}\leq K_1}
      a^{\frac{3dt}{3t+1}+\varepsilon}
      \abs{hk_1}^{\frac{t}{3t+1}+\varepsilon}
      N^{-\frac{2-2^{2-d}}{3t+1}+\varepsilon}
      K_2^{\frac{4t+1}{3t+1}+\varepsilon}
      \right)^{\frac{1}{2^{d-1}}}. 
  \end{gather}

  We continue with part $S_3$ and again use Hölder's inequality to get rid of
  the Fourier coefficients. Following the same arguments and applying Lemma
  \ref{lem:p-norm_Fourier_coefficients} and Lemma \ref{lem:approximation_by_G}
  yields
  \begin{equation}\label{eq:S_3}
    \begin{split}
      \abs{S_3}
      &\ll
      \left(\sum_{\abs{k_1}\leq K_1}\left|
      \widehat{G_{r_1}}(k_1,-h\beta,\delta_1)\right|^{\frac{2^{d-1}}{2^{d-1}-1}}\right)^{\frac{2^{d-1}-1}{2^{d-1}}}
      \left(\sum_{\abs{k_1}\leq K_1}
      \left|\frac{R_2}{N}\right|^{2^{d-1}}\right)^{\frac{1}{2^{d-1}}}\\
      &\ll
      \left(\sum_{\abs{k_1}\leq K_1}\left(r_2\delta_2+r_2^2\delta_2\abs{h\beta\alpha_1-k_1\alpha_1}
      +D_N\left(\left(\alpha_2p(an+b)\right)_{n=1}^N\right)
      \right)^{2^{d-1}}
      \right)^{\frac{1}{2^{d-1}}}.
    \end{split}
  \end{equation}

  Now we want to choose $r_2$, $\delta_2$ and $K_2$ such that $S_1$ dominates $S_3$. Therefore we first set $\rho_2=1+\varepsilon_2$ with $\varepsilon_2=\varepsilon_2(d,t)>0$ sufficiently small and  $r_2>\frac{1}{\varepsilon_2}$. Furthermore we set
  \[
    \delta_2^{-1}=\abs{hk_1}N^{\theta_2}
    \quad\text{and}\quad
    K_2=\abs{hk_1}^{\rho_2}N^{\theta_2}
    \quad\text{with}\quad
    \theta_2=\frac{2-2^{2-d}}{2^{d-1}(3t+1)+4t+1}.
  \]
  Then using Proposition \ref{prop:discrepancy_basis_case} for the discrepancy we get that
  \begin{align*}
    S_3\ll S_1
    \ll
    \left(\sum_{\abs{k_1}\leq K_1}
      a^{\frac{3dt}{3t+1}+\varepsilon}
      \abs{hk_1}^{\frac{5t+1}{3t+1}+\varepsilon}
      N^{-\frac{2^{d-1}(2-2^{2-d})}{2^{d-1}(3t+1)+4t+1  }+\varepsilon}
      \right)^{\frac{1}{2^{d-1}}}.
  \end{align*}

  Summing over $k_1$ yields
  \begin{align*}
    S_1
    \ll
    a^{\frac{3dt}{2^{d-1}(3t+1)}+\varepsilon}
      h^{\frac{5t+1}{2^{d-1}(3t+1)}+\varepsilon}
      N^{-\frac{2-2^{2-d}}{2^{d-1}(3t+1)+4t+1}+\varepsilon}
      K_1^{\frac{8t+2}{2^{d-1}(3t+1)}+\varepsilon}.
  \end{align*}
  Again we set $\rho_1=1+\varepsilon_1$ with $\varepsilon_1=\varepsilon_1(d,t)>0$ sufficiently small and $r_1>\frac{1}{\varepsilon_1}$. Let $\theta_1>0$ be a parameter we choose in an instant. Setting
  \[
    \delta_1^{-1}=hN^{\theta_1}
    \quad\text{and}\quad
    K_1=h^{\rho_1}N^{\theta_1}
  \]
  we get
  \begin{align*}
    S_1
    \ll
    a^{\frac{3dt}{2^{d-1}(3t+1)}+\varepsilon}
      h^{\frac{13t+3}{2^{d-1}(3t+1)}+\varepsilon}
      N^{-\frac{2-2^{2-d}}{2^{d-1}(3t+1)+4t+1}+
      \theta_1\frac{8t+2}{2^{d-1}(3t+1)}+\varepsilon}.
  \end{align*}
  Finally we recall the outer sum over $h$ in \eqref{eq:13} to get
  \begin{gather}\label{eq:h_sum_S_1}
    \sum_{h=1}^H \frac1h S_1
    \ll
    a^{\frac{3dt}{2^{d-1}(3t+1)}+\varepsilon}
      H^{\frac{13t+3}{2^{d-1}(3t+1)}+\varepsilon}
      N^{-\frac{2-2^{2-d}}{2^{d-1}(3t+1)+4t+1}+
      \theta_1\frac{8t+2}{2^{d-1}(3t+1)}+\varepsilon}.
  \end{gather}

  Now we turn our attention to the sum $S_2$. Again by Hölder's inequality we get
  \begin{align*}
    \abs{S_2}
    &\ll
    \left(\sum_{\abs{k_1}\leq K_1}
    \left|\widehat{G_{r_1}}(k_1,-h\beta,\delta_1)\right|^{\frac{2^{d-1}}{2^{d-1}-1}}\right)^{\frac{2^{d-1}-1}{2^{d-1}}}
    \left(\sum_{\abs{k_1}\leq K_1}
    \left|
    \left(\delta_2K_2\right)^{-r_2}\right|^{2^{d-1}}\right)^{\frac{1}{2^{d-1}}}\\
    &\ll
    \left(\sum_{\abs{k_1}\leq K_1}
    \abs{hk_1}^{-2^{d-1}+\varepsilon}\right)^{\frac{1}{2^{d-1}}},
  \end{align*}
  where we have used that
  \[
    (\delta_2K_2)^{-r_2}=\abs{hk_1}^{-(\rho_2-1)r_2}\leq \abs{hk_1}^{-1+\varepsilon}.
  \]
  Summing over $k_1$ yields
  \begin{align*}
    \abs{S_2}
    &\ll
    \left(\abs{h}^{-2^{d}+1+\varepsilon}N^{2^{d-1}-\theta_1(2^{d-1}-1)+\varepsilon}\right)^{\frac{1}{2^{d-1}}}\\
    &\ll
    h^{-\frac{2^d-1}{2^{d-1}}+\varepsilon}N^{-\theta_1\frac{2^{d-1}-1}{2^{d-1}}+\varepsilon}
  \end{align*}
  and finally summing over $h$ in \eqref{eq:13} we get
  \begin{gather}\label{eq:h_sum_S_2}
    \sum_{h=1}^H \frac1h S_2
    \ll
    H^{-\frac{2^d-1}{2^{d-1}}+\varepsilon}N^{-\theta_1\frac{2^{d-1}-1}{2^{d-1}}+\varepsilon}
  \end{gather}
  
  By similar calculations we obtain
  \begin{gather}\label{eq:h_sum_S_4_and_S_5}
    \sum_{h=1}^H \frac1h S_4 \ll H^{-1}
    \quad\text{and}\quad
    \sum_{h=1}^H \frac1h S_5 \ll N^{-\theta_1}.
  \end{gather}

  Now we set $H=a^{-\sigma} N^{\theta_1}$ and choose
  \[
    \sigma=\frac{3dt}{2^{d}(3t+1)+5t+1}
    \quad\text{and}\quad
    \theta_1=\frac{(2-2^{2-d})2^{d-1}(3t+1)}{(2^{d-1}(3t+1)+21t+5)(2^{d}(3t+1)+4t+1)}.
  \]
  Replacing the four parts \eqref{eq:h_sum_S_1}, \eqref{eq:h_sum_S_2} and \eqref{eq:h_sum_S_4_and_S_5} in the Erd\H{o}s-Turán inequality \eqref{eq:13} proves the proposition.
\end{proof}

\section{Proof of Theorem \ref{thm:main}}

As indicated in the introduction we link the estimate of along the arithmetic
progressions with the discrepancy estimate we obtained in the section above. To
this end we fix an arbitrary arithmetic progression, \textit{i.e.} let
$a\in\NN^*$, $b\in\ZZ$ and $1\leq x\leq N$. Then
\begin{align*}
  \abs{\sum_{am+b\leq x}e_{am+b}}
  &\leq\abs{\sum_{am+b\leq x} \left(1\!\!1_{\{f(am+b)\}<\tfrac12}-\frac12\right)}
    +\abs{\sum_{am+b\leq x} \left(1\!\!1_{\{f(am+b)\}\geq\tfrac12}-\frac12\right)}\\
  &\leq 2MD_M\left(\left(\beta\left\lfloor\alpha_1\left\lfloor\alpha_2 p(am+b)
  \right\rfloor\right\rfloor\right)_{m=1}^M\right),
\end{align*}
where $M=\left\lfloor\frac{x-b}{a}\right\rfloor$.

Now we distinguish two cases according to the size of $a$, namely if
\[
  a\leq x^{\frac{1}{(2^{d}(3t+1)+21t+5)^2}}
\]
or not. In the first case we may apply Proposition \ref{prop:discrepancy_second_iteration} and get that there exists $\eta_1=\eta_1(d,t)$ such that
\[
  \abs{\sum_{am+b\leq x}e_{am+b}}
  \ll x^{1-\eta_1}.
\]

In the second case we have
\[
  a>x^{\frac{1}{(2^{d}(3t+1)+21t+5)^2}}
\]
and a trivially estimate yields
\[
  \abs{\sum_{am+b\leq x}e_{am+b}}
  \ll x^{1-\frac{1}{(2^{d}(3t+1)+21t+5)^2}}.
\]

Thus the theorem is proven by these two cases.

\section*{Acknowledgment}
The first author is supported by project ANR-18-CE40-0018 funded by
the French National Research Agency.


\begin{thebibliography}{10}

  \bibitem{aistleitner2013:limit_distribution_well}
  C.~Aistleitner.
  \newblock On the limit distribution of the well-distribution measure of random binary sequences.
  \newblock {\em J. Th\'eor. Nombres Bordeaux}, 25(2):245--259, 2013.
  
  \bibitem{bergelson_ha_son2021:extension_weyls_equidistribution}
  V.~Bergelson, I.~J. H\aa~land Knutson, and Y.~Son.
  \newblock An extension of {W}eyl's equidistribution theorem to generalized polynomials and applications.
  \newblock {\em Int. Math. Res. Not. IMRN}, (19):14965--15018, 2021.
  
  \bibitem{bergelson_kolesnik_madritsch+2014:uniform_distribution_prime}
  V.~Bergelson, G.~Kolesnik, M.~Madritsch, Y.~Son, and R.~Tichy.
  \newblock Uniform distribution of prime powers and sets of recurrence and van der {C}orput sets in {$\Bbb{Z}^k$}.
  \newblock {\em Israel J. Math.}, 201(2):729--760, 2014.
  
  \bibitem{bergelson_kolesnik_son2019:uniform_distribution_subpolynomial}
  V.~Bergelson, G.~Kolesnik, and Y.~Son.
  \newblock Uniform distribution of subpolynomial functions along primes and applications.
  \newblock {\em J. Anal. Math.}, 137(1):135--187, 2019.
  
  \bibitem{drmota_tichy1997:sequences_discrepancies_and}
  M.~Drmota and R.~F. Tichy.
  \newblock {\em Sequences, discrepancies and applications}, volume 1651 of {\em Lecture Notes in Mathematics}.
  \newblock Springer-Verlag, Berlin, 1997.
  
  \bibitem{graham_kolesnik1991:van_der_corputs}
  S.~W. Graham and G.~Kolesnik.
  \newblock {\em van der {C}orput's method of exponential sums}, volume 126 of {\em London Mathematical Society Lecture Note Series}.
  \newblock Cambridge University Press, Cambridge, 1991.
  
  \bibitem{granville_ramare1996:explicit_bounds_on}
  A.~Granville and O.~Ramar\'{e}.
  \newblock Explicit bounds on exponential sums and the scarcity of squarefree binomial coefficients.
  \newblock {\em Mathematika}, 43(1):73--107, 1996.
  
  \bibitem{hofer_ramare2016:discrepancy_estimates_some}
  R.~Hofer and O.~Ramar\'{e}.
  \newblock Discrepancy estimates for some linear generalized monomials.
  \newblock {\em Acta Arith.}, 173(2):183--196, 2016.
  
  \bibitem{kuipers_niederreiter1974:uniform_distribution_sequences}
  L.~Kuipers and H.~Niederreiter.
  \newblock {\em Uniform distribution of sequences}.
  \newblock Wiley-Interscience [John Wiley \& Sons], New York, 1974.
  \newblock Pure and Applied Mathematics.
  
  \bibitem{madritsch_rivat_tichy2024:finite_pseudorandom_binary}
  M.~G. Madritsch, J.~Rivat, and R.~F. Tichy.
  \newblock On finite pseudorandom binary sequences: functions from a {H}ardy field.
  \newblock {\em Acta Math. Hungar.}, 174(1):121--137, 2024.
  
  \bibitem{madritsch_tichy2016:dynamical_systems_and}
  M.~G. Madritsch and R.~F. Tichy.
  \newblock Dynamical systems and uniform distribution of sequences.
  \newblock In {\em From arithmetic to zeta-functions}, pages 263--276. Springer, [Cham], 2016.
  
  \bibitem{madritsch_tichy2019:multidimensional_van_der}
  M.~G. Madritsch and R.~F. Tichy.
  \newblock Multidimensional van der {C}orput sets and small fractional parts of polynomials.
  \newblock {\em Mathematika}, 65(2):400--435, 2019.
  
  \bibitem{mauduit_rivat_sarkoezy2002:pseudo_random_properties}
  C.~Mauduit, J.~Rivat, and A.~S{\'a}rk{\"o}zy.
  \newblock On the pseudo-random properties of {$n^c$}.
  \newblock {\em Illinois J. Math.}, 46(1):185--197, 2002.
  
  \bibitem{mauduit_sarkoezy1997:finite_pseudorandom_binary}
  C.~Mauduit and A.~S{\'a}rk{\"o}zy.
  \newblock On finite pseudorandom binary sequences. {I}. {M}easure of pseudorandomness, the {L}egendre symbol.
  \newblock {\em Acta Arith.}, 82(4):365--377, 1997.
  
  \bibitem{mauduit_sarkoezy2000:finite_pseudorandom_binary}
  C.~Mauduit and A.~S{\'a}rk{\"o}zy.
  \newblock On finite pseudorandom binary sequences. {V}. {O}n {$(n\alpha)$} and {$(n^2\alpha)$} sequences.
  \newblock {\em Monatsh. Math.}, 129(3):197--216, 2000.
  
  \bibitem{mauduit_sarkoezy2000:finite_pseudorandom_binary2}
  C.~Mauduit and A.~S{\'a}rk{\"o}zy.
  \newblock On finite pseudorandom binary sequences. {VI}. {O}n {$(n^k\alpha)$} sequences.
  \newblock {\em Monatsh. Math.}, 130(4):281--298, 2000.
  
  \bibitem{montgomery1994:ten_lectures_on}
  H.~L. Montgomery.
  \newblock {\em Ten lectures on the interface between analytic number theory and harmonic analysis}, volume~84 of {\em CBMS Regional Conference Series in Mathematics}.
  \newblock Published for the Conference Board of the Mathematical Sciences, Washington, DC, 1994.
  
  \bibitem{mukhopadhyay_ramare_viswanadham2018:discrepancy_estimates_generalized}
  A.~Mukhopadhyay, O.~Ramar\'{e}, and G.~K. Viswanadham.
  \newblock Discrepancy estimates for generalized polynomials.
  \newblock {\em Monatsh. Math.}, 187(2):343--356, 2018.
  
  \bibitem{nathanson1996:additive_number_theory}
  M.~B. Nathanson.
  \newblock {\em Additive number theory}, volume 164 of {\em Graduate Texts in Mathematics}.
  \newblock Springer-Verlag, New York, 1996.
  \newblock The classical bases.
  
  \bibitem{rivat_tenenbaum2005:constantes_derd_os}
  J.~Rivat and G.~Tenenbaum.
  \newblock Constantes d'{E}rd{\H o}s-{T}ur\'an.
  \newblock {\em Ramanujan J.}, 9(1-2):111--121, 2005.
  
  \end{thebibliography}

\end{document}